\documentclass[smallextended]{svjour3}
\smartqed
\usepackage{graphicx}
\usepackage{amssymb,amsmath}
\usepackage{hyperref}
\usepackage{pst-all}
\newtheorem{thm}{Theorem}

\newtheorem{lem}[thm]{Lemma}
\newtheorem{prop}[thm]{Proposition}

\allowdisplaybreaks[3]

\begin{document}

\title{Chromatic polynomials of complements of bipartite graphs}
\author{Adam Bohn}
\institute{Adam Bohn \at
              School of Mathematical Sciences \\
              Queen Mary, University of London \\
              Mile End Road \\
              London E1 4NS, U.K.\\
\email{a.bohn@qmul.ac.uk}
}

\maketitle

\begin{abstract}
Bicliques are complements of bipartite graphs; as such each consists of two cliques joined by a number of edges. In this paper we study algebraic aspects of the chromatic polynomials of these graphs.  We derive a formula for the chromatic polynomial of an arbitrary biclique, and use this to give certain conditions under which two of the graphs have chromatic polynomials with the same splitting field.  Finally, we use a subfamily of bicliques to prove the cubic case of the $\alpha +n$ conjecture, by showing that for any cubic integer $\alpha$, there is a natural number $n$ such that $\alpha +n$ is a chromatic root.\\
\end{abstract}

\section{Introduction}

If $q$ is a positive integer, then a \emph{proper $q$-colouring} of a graph $G$ is a function from the vertices of $G$ to a set of $q$ colours, with the property that adjacent vertices receive different colours.  The \emph{chromatic polynomial} $P_G(x)$ of $G$ is the unique monic polynomial which, when evaluated at $q$, gives the number of proper $q$-colourings of G.  A \emph{chromatic root} of $G$ is a zero of $P_G(x)$.  The chromatic polyomial has been the subject of much study; see \cite{read:introduction} for a comprehensive introduction.

This paper is a study of the chromatic polynomials of a family of simple graphs we shall call \textit{bicliques}.  This family is easily described: each member is simply the complement of a bipartite graph, consisting of two cliques joined by a number of edges.  When we need to be more specific, we shall refer to a biclique in which the two cliques are of size $j$ and $k$ as a \emph{$(j,k)$-biclique.}  By convention, $k$ will be greater than or equal to $j$, and we shall refer to the edges between the two cliques as \textit{bridging edges.}

The present study of bicliques was originally undertaken with the aim of extending the proof of Cameron's ``$\alpha + n$ conjecture,'' which suggests that for any algebraic integer $\alpha$ there is a positive integer $n$ such that $\alpha + n$ is a chromatic root.  Cameron and Morgan \cite{pjc11} proved the conjecture in the quadratic case, by showing that any quadratic integer is an integer shift of a chromatic root of a $(2,k)$-biclique, but no further progress has subsequently been made.

In section 2, we give a simple construction of the chromatic polynomial of an arbitrary biclique.  We then use this construction in section 3 to examine relations between bicliques having chromatic polynomials with the same splitting field.  In section 4 we justify the original motivation for studying bicliques, by applying them to prove the cubic case of the $\alpha + n$ conjecture, and finally we remark on the potential suitability of bicliques for proving the general conjecture.

\section{Chromatic polynomials of bicliques}\label{sec:poly}

A \emph{matching} of a graph is a set of edges of that graph, no two of which are incident to the same vertex.  When we refer to the size of a matching, we refer to the number of edges in that matching; an \emph{$i$-matching} is a matching of size $i$.  Let $m_G^i$ be the number of $i$-matchings of a graph $G$; then the \emph{matching numbers} of $G$ are the elements of the sequence $(m_G^0,m_G^1,m_G^2,\ldots)$.  We will write that two graphs are \emph{matching equivalent} if they have the same matching numbers.

Now, for some positive integers $j$ and $k$, let $G$ be a $(j,k)$-biclique, and let $\bar G$ be the complement of $G$ (obtained by replacing edges of $G$ with non-edges, and vice-versa).  Then $\bar G$ is a subgraph of the complete bipartite graph $K_{j,k}$.  We shall construct the chromatic polynomial of $G$ by considering matchings of $\bar G$.

Given some matching of $\bar G$, partition the vertices of $G$ such that two vertices are contained in the same part if and only if the corresponding vertices of $\bar G$ are joined by an element of the matching.  Then, by assigning a different colour to each part of this partition, we obtain a proper colouring of $G$.  Conversely, any proper colouring of $G$ corresponds to a partition induced by some matching of $\bar G$.  Thus we can compute the chromatic polynomial of $G$ by counting $x$-colourings of partitions induced by matchings of $\bar G$, as follows.

Let $(x)_k$ denote the falling factorial $x(x-1)\cdots(x-k+1)$.  If each part receives a different colour, then there are $(x)_{j+k-i}$ ways of assigning $x$ colours to a partition induced by an $i$-matching of $\bar G$ (as any such partition consists of $j+k-i$ parts).  Thus:
\begin{equation}\label{eqn:formula}
P_G(x)=\sum_M(x)_{j+k-|M|},
\end{equation}
where the sum is over all possible matchings $M$ of $\bar G$.  Note that this construction in fact gives us the unique decomposition of $P_G(x)$ into a sum of chromatic polynomials of complete graphs (the chromatic polynomial of the $n$-vertex complete graph $K_n$ being $(x)_n$).

Now suppose that, for some $1\leq p\leq j$, there are $p$ vertices in the $j$-clique of $G$ which are adjacent to every vertex of the $k$-clique.  Then these $p$ vertices are each adjacent to every other vertex of the graph.  Thus, counting proper $x$-colourings of $G$, we have that there are $(x)_p$ ways in which to colour these $p$ vertices, and $x-p$ colours remaining with which to colour the remaining vertices.  So the chromatic polynomial of $G$ will be of the form:
\[P_G(x)=(x)_pP_H(x-p),\]
where $H$ is the $(j-p,k)$-biclique obtained from $G$ by deleting each of the $p$ vertices and all incident edges.  A similar situation arises if some vertices of the $k$-clique are adjacent to every vertex of the $j$-clique.  As we are concerned with algebraic properties of the chromatic polynomial, we shall discount these cases, and assume that no vertex of $G$ is connected to every other vertex of the graph.

With this condition on $G$, it is not difficult to see that $\bar G$ will always have a single $0$-matching, along with at least one $i$-matching for all $0<i \leq j$, and that no larger matchings are possible.  Hence we have:
\begin{equation*}
P_G(x)=\sum_{i=0}^jm_{\bar G}^i(x)_{j+k-i},
\end{equation*}
where $m_{\bar G}^0=1$, and in general $m_{\bar G}^i$ is a positive integer.  Thus the chromatic polynomial of $G$ is a product of $(x)_k$ with a (usually irreducible) degree $j$ factor $g(x)$ of the form:
\begin{equation}\label{eqn:decomp}
g(x)=\sum_{i=0}^jm_{\bar G}^i(x-k)_{j-i}.
\end{equation}
It is this factor which will be our main object of study, and we will henceforth refer to it as the ``interesting factor'' of $P_G(x)$.

\section{Chromatic splitting field-equivalent bicliques}

We define the \emph{chromatic splitting field }of a graph $G$ to be the splitting field of the chromatic polynomial of $G$ --- that is, the smallest field extension of $\mathbb Q$ in which $P_G(x)$ factorises completely into linear factors.  We can indirectly study the algebraic diversity of chromatic polynomials of bicliques by investigating how common it is for two graphs in the family to have the same chromatic splitting field.  In doing so we may restrict our attention to the interesting factors of their chromatic polynomials, as multiplication by linear factors does not change a polynomial's splitting field.

Clearly, in order for two bicliques to share the same chromatic splitting field, the interesting factors of their chromatic polynomials must be of the same degree.  That is, there must be positive integers $j,k_G$ and $k_H$ such that one graph is a $(j,k_G)$-biclique and the other is a $(j,k_H)$-biclique.  The simplest way in which two such graphs might share the same chromatic splitting field is if the interesting factor of one chromatic polynomial is an integer shift of that of the other, and it is easy to show that this is always the case when the graphs' complements are matching equivalent.

\begin{prop}
Let $j,k_G$ and $k_H$ be positive integers with $j\leq k_G\leq k_H$, and let $G$ and $H$ be, respectively, a $(j,k_G)$-biclique and a $(j,k_H)$-biclique.  Denote by $g(x)$ and $h(x)$ the degree $j$ interesting factors of $P_G(x)$ and $P_H(x)$.  If $\bar G$ and $\bar H$ are matching equivalent then $g(x)=h(x+k_H-k_G)$.
\end{prop}

\begin{proof}
Suppose that $\bar G$ and $\bar H$ are matching equivalent.  As $m_{\bar H}^i=m_{\bar G}^i$ for all $0\leq i\leq j$, we have from (\ref{eqn:decomp}) that
\[g(x)=\sum_{i=0}^jm_{\bar G}^i(x-k_G)_{j-i}\]
and
\[h(x)=\sum_{i=0}^jm_{\bar G}^i(x-k_H)_{j-i},\]
so clearly $g(x)=h(x+k_H-k_G)$.
\end{proof}
\qed

Now, let $g(x)$ and $h(x)$ be polynomials of equal degree, and suppose there exists some integer $c$ such that $g(x)=(-1)^jh(-x+c)$; we will describe this scenario by writing that $g(x)$ and $h(x)$ are \emph{reflections} of each other.  Clearly two such polynomials have the same splitting field.  It turns out to be surprisingly common for two bicliques to have chromatic polynomials with interesting factors which are reflections of each other.  The following theorem gives a necessary and sufficient condition under which this occurs.

\begin{thm}\label{thm:match}
Let $j,k_G$ and $k_H$ be positive integers satisfying $j\leq k_G \leq k_H$, and let $G$ and $H$ be, respectively, a $(j,k_G)$-biclique and a $(j,k_H)$-biclique, having chromatic polynomials $P_G(x)=(x)_{k_G}g(x)$ and $P_H(x)=(x)_{k_H}h(x)$.  Then $g(x)=(-1)^jh(-x+c)$ for some integer $c$ if and only if
\begin{equation}\label{eqn:general}
m_{\bar G}^i=\sum_{l=0}^i(-1)^lm_{\bar H}^l{j-l \choose j-i}(k_G+k_H+j-c-l-1)_{i-l},
\end{equation}
for all $0\leq i\leq j$.
\end{thm}

\begin{proof}
Suppose that the stated condition holds. From (\ref{eqn:decomp}), we have that:
\[g(x)=\sum_{i=0}^jm_{\bar G}^i(x-k_G)_{j-i},\]
and so substituting for $m_{\bar G}^i$ we get:
\begin{align*}
g(x)=&\sum_{i=0}^j\sum_{l=0}^i(-1)^lm_{\bar H}^l{j-l \choose j-i}(k_G+k_H+j-c-l-1)_{i-l}(x-k_G)_{j-i}\\
    =&\sum_{l=0}^j(-1)^lm_{\bar H}^l\sum_{i=0}^{j-l}{j-l \choose j-i-l}(k_G+k_H+j-c-l-1)_i(x-k_G)_{j-i-l}\\
    =&\sum_{l=0}^j(-1)^lm_{\bar H}^l\sum_{i=0}^{j-l}{j-l \choose i}(k_G+k_H+j-c-l-1)_i(x-k_G)_{j-i-l}\\
    =&\sum_{l=0}^j(-1)^lm_{\bar H}^l(x+k_H+j-c-l-1)_{j-l}\\
    =&\sum_{l=0}^j(-1)^lm_{\bar H}^l(-1)^{j-l}(-x-k_H+c)_{j-l}\\
    =&(-1)^j\sum_{l=0}^jm_{\bar H}^l(-x-k_H+c)_{j-l}\\
    =&(-1)^jh(-x+c).
\end{align*}
The converse is proved by simply reversing these steps.

\end{proof}
\qed

Theorem \ref{thm:match} is quite technical, and does not appear to significantly increase our intuitive understanding of which pairs of bicliques have chromatic polynomials with reflected interesting factors.  However, from it we are able to deduce an interesting consequence.  We will require the following lemma, which is essentially a specialisation or rephrasing of similar results by, among others, Riordan \cite{Riordan:Introduction}, Farrell and Whitehead \cite{Farrell:Connections} and Zaslavsky \cite{Zaslavsky:Complementary}.  For the sake of completion, we shall include a full proof here.

\begin{lem}\label{lem:match}
Let $\bar G$ be a spanning subgraph of the complete bipartite graph $K_{j,k}$, where $j\leq k$, and let $\bar H$ be the complement of $G$ in $K_{j,k}$.  Then:
\[m_{\bar G}^i=\sum_{l=0}^i(-1)^lm_{\bar H}^l{j-l \choose j-i}(k-l)_{i-l}.\]
\end{lem}

\begin{proof}
Given some $0\leq i\leq j,$ let $\mathcal{X}^i$ be the set of all $i$-matchings of $K_{j,k}$, and for each edge $e$ of $K_{j,k}$ let $\mathcal{A}_e^i\subset \mathcal{X}^i$ be the collection of those $i$-matchings containing $e$.  Label the edges of $\bar H$ as $\{1,2,\ldots ,m\}$.  The $i$-matchings of $\bar G$ are simply those $i$-matchings of $K_{j,k}$ not containing any edge of $\bar H$, and so the number of $i$-matchings of $\bar G$ is:
\[m_{\bar G}^i=\left|\mathcal{X}^i\setminus \bigcup_{e=1}^m\mathcal{A}_e^i\right|.\]
By the Principle of Inclusion-Exclusion, the right-hand side of this equation is precisely:
\begin{equation}\label{eqn:inc-ex}
\sum_{I\subseteq \{1,\ldots ,m\}}(-1)^{|I|}\left|\bigcap_{e\in I}\mathcal{A}_e^i\right|.
\end{equation}

Now, note that the $i$-matchings contained in $\bigcap_{e\in I}\mathcal{A}_e^i$ are precisely those $i$-matchings of $K_{j,k}$ containing every $e\in I$.  Furthermore, if $I\subseteq \{1,\ldots ,m\}$ is not a matching, or else has size greater than $i$, then the number of $i$-matchings of $K_{j,k}$ containing every $e\in I$ is zero; and if $I$ is a matching of size less than or equal to $i$ then the number of $i$-matchings of $K_{j,k}$ containing every $e\in I$ depends only on the cardinality of $I$ (not on its precise edge-content).  Thus (\ref{eqn:inc-ex}) is equivalent to the alternating sum, over all $0\leq l\leq i$, of the product of the number of $i$-matchings of $K_{j,k}$ containing a given $l$-matching, with the number of possible $l$-matchings of $\bar H$.

We can count the $i$-matchings of $K_{j,k}$ containing a given $l$-matching as follows: the $l$ edges of the $l$-matching join $l$ vertices on the ``$j$-side'' of $K_{j,k}$ to $l$ vertices on the ``$k$-side''.  From the remaining $j-l$ vertices on the $j$-side, we have a choice of $i-l$ to be incident to the extra edges in our desired $i$-matching.  For each such choice of $i-l$ vertices we then have $(k-l)_{i-l}$ ways to choose their neighbours on the $k$-side.  So we have that the number of $i$-matchings of $K_{j,k}$ containing a given $l$-matching is:
\[{j-l \choose i-l}(k-l)_{i-l}={j-l \choose j-i}(k-l)_{i-l}.\]

The number of possible $l$-matchings of $\bar H$ is simply $m_{\bar H}^l$, and so putting this all together we obtain from (\ref{eqn:inc-ex}) the desired expression:
\begin{equation}\label{eqn:match}
m_{\bar G}^i=\sum_{l=0}^i(-1)^lm_{\bar H}^l{j-l \choose j-i}(k-l)_{i-l}.
\end{equation}
\end{proof}
\qed

Now, let $G$ be a $(j,k)$-biclique, and let $H$ be the graph which is obtained from $G$ by replacing all bridging edges by non-edges, and vice-versa.  We will refer to $H$ as being the \emph{complementary partner} of $G$, or else to $G$ and $H$ as being \emph{complementary $(j,k)$-bicliques.} We can now prove that the members of such a pair have chromatic polynomials with reflected interesting factors.

\begin{prop}\label{prop:comp}
For some positive integers $j$ and $k$ with $j\leq k$ let $G$ and $H$ be complementary $(j,k)$-bicliques, and let $g(x)$ and $h(x)$ be the interesting factors of $P_G(x)$ and $P_H(x)$ respectively. Then:
\[g(x)=(-1)^jh(-x+j+k-1).\]
\end{prop}

\begin{proof}
Note that $\bar G$ and $\bar H$ complement each other inside the complete bipartite graph $K_{j,k}$.  Hence, by Lemma \ref{lem:match}:
\[m_{\bar G}^i=\sum_{l=0}^i(-1)^lm_{\bar H}^l{j-l \choose j-i}(k-l)_{i-l}.\]
This expression is simply (\ref{eqn:general}) with $k_G=k_H=k$ and $c=j+k-1$, so the result follows from Theorem \ref{thm:match}.
\end{proof}
\qed

We have chosen to present this specific case in full as it is particularly aesthetically pleasing, however it should be noted that there are many more pairs of bicliques other than complementary pairs which have chromatic polynomials with reflected interesting factors.  With the final section of this paper in mind, we will prove one of these for the case $j=3$, and mention two more.

We can parameterise a $(3,k)$-clique $G$ in the following way: label the three vertices in the $3$-clique $v_1,v_2$ and $v_3$; let $a, b$ and $c$ represent the number of neighbours of $v_1,v_2$ and $v_3$ respectively in the $k$-clique; and let $d,e$ and $f$ represent the number of vertices in the $k$-clique joined to both $v_2$ and $v_3$, both $v_1$ and $v_3$, and both $v_1$ and $v_2$ respectively.  Then the 6-tuple $(a,b,c,d,e,f)$ completely describes $G$, and we will simply write that $G=(a,b,c,d,e,f)$.
  
It will be helpful for what follows to point out some properties of the complement of this graph.  So note that the order of $\bar G$ is
\[|V(\bar G)|=a+b+c+d+e+f+3;\]
that the number of edges of $\bar G$ is
\[|E(\bar G)|=2a+2b+2c+d+e+f;\]
and that the degrees of $v_1,v_2$ and $v_3$ in $\bar G$ are $(b+c+d),$ $(a+c+e)$ and $(a+b+f)$ respectively.

\begin{prop}\label{prop:reflect1}
Let $r,s,t,u$ and $v$ be non-negative integers satisfying $u+v=4t-r-s+3$, and let $G=(r,s,t,t,t,u)$ and $H=(r,s,t,t,t,v)$ be $(3,k)$-bicliques, having chromatic polynomials with interesting factors $g(x)$ and $h(x)$ respectively.  Then:
\[g(x)=-h(-x+6t+4).\]
\end{prop}

\begin{proof}
First note that the number of vertices connected only to $v_1,v_2$ and $v_3$ in $\bar G$ are $t,t$ and $u$ respectively, and the corresponding values for $\bar H$ are $t,t$ and $v$.  Furthermore, in both $\bar G$ and $\bar H$ the number of vertices connected to both $v_2$ and $v_3$, both $v_1$ and $v_3$, and both $v_1$ and $v_2$ are $r,s$ and $t$ respectively.

Now, it is clear that $m_{\bar G}^0=m_{\bar H}^0=1$.  The matching numbers $m_{\bar G}^1$ and $m_{\bar H}^1$ can also be easily found, as they are simply the numbers of edges of $\bar G$ and $\bar H$, that is:
\[m_{\bar G}^1=2r+2s+4t+u\]
and
\[m_{\bar H}^1=2r+2s+4t+v.\]

Now, let $B$ be the subgraph of $\bar G$ resulting from the removal of those edges incident only to $v_3$ and no other $v_i$, and let $A$ be the subgraph of $B$ obtained by removing $v_3$ and all remaining incident edges.  For $l=2$ or $l=3$, we can split each $l$-matching of $\bar G$ into one of two groups: those containing one of the $u$ edges incident to $v_3$ and no other $v_i$, and those not containing such an edge.  The former consists of every $l$-matching which is a union of an $(l-1)$-matching of $A$ with one of the $u$ edges in question; the latter are simply the $l$-matchings of $B$.  So we have, for $l=2$ or $3$:
\begin{equation}\label{eqn:matchsplit}
m_{\bar G}^l=um_A^{l-1}+m_B^l.
\end{equation}

By enumerating the edges of $A$ we immediately find that $m_A^1=r+s+4t$.  The 2-matchings of $A$ can be counted by multiplying the number of vertices adjacent to $v_1$ by the number adjacent to $v_2$, and subtracting $t$ (as there are $t$ vertices adjacent to both $v_1$ and $v_2$), giving us:
\[m_A^2=(r+2t)(s+2t)-t=rs+2rt+2st+4t^2-t.\]

It remains to find $m_B^2$ and $m_B^3$.  The 2-matchings can be found by subtracting the number of pairs of edges of $B$ which are incident to a common vertex from the total number of pairs of edges of $B$.   The total number of pairs of edges of $B$ is:
\[{|E(B)| \choose 2}={2r+2s+4t \choose 2}=\frac{1}{2}(2r+2s+4t)(2r+2s+4t-1),\]
and if we represent by $d(v)$ the degree of the vertex $v$, then the number of pairs which are incident to a common vertex is:
\begin{align*}
\sum_{v\in V(B)}{d(v) \choose 2}&=\frac{1}{2}\left(\sum_{v\in V(B)}d(v)^2-d(v)\right)\\
                                &=\frac{1}{2}\left(-2|E(B)|+\sum_{v\in V(B)}d(v)^2\right)\\
                                &=s^2+2st+4t^2+r^2+2rt+rs-t.
\end{align*}

Thus:\begin{align*}
m_B^2&=\frac{1}{2}(2r+2s+4t)(2r+2s+4t-1)\\
     & -(s^2+2st+4t^2+r^2+2rt+rs-t)\\
     &=r^2+3rs+6rt-r+s^2+6st-s+4t^2-t.
\end{align*}

To count 3-matchings, note that the total number of choices of 3 edges such that one is incident to each of the $v_i$ is
\[(s+2t)(r+2t)(r+s).\]
From this we will need to subtract: 
\begin{enumerate}
  \item the $t(r+s)$ 3-matchings in which the chosen edges incident to $v_1$ and $v_2$ share a common endpoint;
  \item the $s(r+2t)$ 3-matchings in which the chosen edges incident to $v_1$ and $v_3$ share a common endpoint; and: 
  \item the $r(s+2t)$ 3-matchings in which the chosen edges incident to $v_2$ and $v_3$ share a common endpoint.
\end{enumerate}

This gives us:
\[m_B^3=(s+2t)(r+2t)(r+s)-t(r+s)-s(r+2t)-r(s+2t).\]

Finally, substituting all of these into (\ref{eqn:matchsplit}), we have:
\[m_{\bar G}^2=u(r+s+4t)+r^2+3rs+6rt-r+s^2+6st-s+4t^2-t,\]
and
\begin{align*}
m_{\bar G}^3&=u(rs+2rt+2st+4t^2-t)\\
            &+(s+2t)(r+2t)(r+s)-t(r+s)-s(r+2t)-r(s+2t).
\end{align*}

The matching numbers of $\bar H$ can now be derived by simply substituting $v$ for $u$ in these expressions, giving:
\[m_{\bar H}^2=v(r+s+4t)+r^2+3rs+6rt-r+s^2+6st-s+4t^2-t,\]
and
\begin{align*}
m_{\bar H}^3&=v(rs+2rt+2st+4t^2-t)\\
            &+(s+2t)(r+2t)(r+s)-t(r+s)-s(r+2t)-r(s+2t).
\end{align*}
It is now simple, if rather tedious, to verify that for each $0\leq i\leq 3$:
\[m_{\bar G}^i=\sum_{l=0}^i(-1)^lm_{\bar H}^l{3-l \choose 3-i}(4t+r+s+1-l)_{i-l}.\]
This equation is simply (\ref{eqn:general}) with the substitutions:
\begin{align*}
k_G&=r+s+3t+u\\
k_H&=r+s+3t+v\\
j&=3\\
c&=6t+4,
\end{align*}
which means that $G$ and $H$ satisfy the conditions of Theorem \ref{thm:match}, and
\[g(x)=-h(-x+6t+4).\]
\end{proof}
\qed

The proofs of the following two results proceed in exactly the same way as that of Proposition \ref{prop:reflect1}, and so we will spare the reader the details.  In both cases $G$ and $H$ are $(3,k)$-bicliques having chromatic polynomials with interesting factors $g(x)$ and $h(x)$ respectively.

\begin{prop}
Let $r,s,t,u$ and $v$ be non-negative integers satisfying $u+v=4t-2r+4$. If $G=(r,r+s-1,t,t,s+t,u)$ and $H=(r,r+s-1,t,t,s+t,v),$ then:
\[g(x)=-h(-x+2s+6t+4).\]
\end{prop}

\begin{prop}
Let $r,s,t,u$ and $v$ be non-negative integers satisfying $u+v=4s-2r+t^2+2t+4.$  If $G=(r,r,s,s+{t+1 \choose 2},s+{t+2 \choose 2},u)$ and $H=(r,r,s,s+{t+1 \choose 2},s+{t+2 \choose 2},v),$ then:
\[g(x)=-h(-x+6s+2t^2+4t+6).\]
\end{prop}

This is just a sample; there are likely to be more such relations for the relatively simple case $j=3$, and no doubt many more for larger $j$.  Even the few examples we have presented do however suggest patterns which invite further investigation.  It would be desirable to find a more graph-theoretic classification of pairs of bicliques having chromatic polynomials which are related by a reflection than that given by Theorem \ref{thm:match}.  In addition, note that the pivotal relation between the chromatic polynomials of these graphs and matchings of their complements in fact holds for any triangle-free graph (see \cite{Farrell:Connections} for a proof), raising the possibility that our results might generalise to larger classes of graphs.

We finish this section with an observation of a link between proper colourings and acyclic orientations of these special pairs of bicliques.  For some positive integers $j$ and $k$ satisfying $j\leq k$ let $G$ and $H$ be $(j,k)$-bicliques having chromatic polynomials with interesting factors $g(x)$ and $h(x)$ respectively, and suppose that $g(x)=(-1)^jh(-x+c)$ for some integer $c$.  Then:
\[g(x+c)=(-1)^jh(-x),\]
and so
\[P_G(x+c)=(-1)^j\frac{(x+c)_k}{(-x)_k}P_H(-x)=(-1)^{j+k}\frac{(x+c)_k}{(x+k-1)_k}P_H(-x).\]
Evaluating this equation at $x=1$ gives:
\begin{equation}\label{eqn:acyclic}
P_{G}(c+1)={c+1 \choose k}(-1)^{j+k}P_{H}(-1).
\end{equation}
Now, Stanley \cite{Stanley:Acyclic} showed that $(-1)^nP_G(-1)$ is the number of acyclic orientations of an $n$-vertex graph $G$. Thus (\ref{eqn:acyclic}) implies that the number of proper $(c+1)$-colourings of $G$ is $c+1 \choose k$ times the number of acyclic orientations of $H$.  In particular, Proposition \ref{prop:comp} implies that the number of proper $(j+k)$-colourings of a $(j,k)$-biclique is $j+k \choose k$ times the number of acyclic orientations of its complementary partner.  It seems likely that a combinatorial proof of this result could be found, which might shed some more light on the results we have presented in this section.

\section{Cubic integers as chromatic roots}

In \cite{pjc11} it was conjectured that for every algebraic integer $\alpha$ there is a natural number $n$ such that $\alpha +n$ is a chromatic root. The authors of that paper proved this conjecture for quadratic integers, but it has remained unresolved for algebraic numbers of higher degree.  In this section we prove the conjecture for the case of cubic integers by showing that, given any cubic integer, there is a $(3,k)$-biclique having a chromatic root which is an integer shift of it.

Let $G$ be a $(3,k)$-clique.  Using the same parameterisation as in the last section, we have $G=(a,b,c,d,e,f).$  We will first need to use these 6 parameters to provide an alternative construction of $P_G(x)$ from that given in \S\ref{sec:poly}.

As established previously, $P_G(x)$ is a product of $(x)_k$ with a cubic ``interesting factor'' $g(x)$.  Observe that $(x)_k$ is the number of ways to properly $x$-colour the $k$-clique of $G$; thus we can view $g(x)$ as an expression for the number of proper $x$-colourings of the $3$-clique.  We can construct this expression independently of the rest of the polynomial using the Principle of Inclusion-Exclusion, as follows.

Assuming there are no edges between any of the $\{v_i\}$, the number of ways to properly $x$-colour them is:
\begin{equation}\label{eqn:noedges}
(x-a-e-f)(x-b-d-f)(x-c-d-e).
\end{equation}
From these we must subtract those colourings in which two vertices receive the same colour.  There are, for example:
\[(x-a-b-d-e-f)(x-c-d-e)\]
ways in which to properly $x$-colour the 3 vertices such that $v_1$ and $v_2$ receive the same colour.  The other two pairs provide similar expressions, giving us three to subtract from (\ref{eqn:noedges}).  Finally we consider those colourings in which all three vertices receive the same colour.  The number of these is, simply:
\[(x-a-b-c-d-e-f).\]
As we have effectively discounted these three times in the previous step, we must add this expression twice.  Our final interesting factor is:
\begin{eqnarray}
g(x)&=&(x-a-e-f)(x-b-d-f)(x-c-d-e)\label{cp}\\
&-&\nonumber(x-a-b-d-e-f)(x-c-d-e)\\
&-&\nonumber(x-a-c-d-e-f)(x-b-d-f)\\
&-&\nonumber(x-b-c-d-e-f)(x-a-e-f)+2(x-a-b-c-d-e-f).
\end{eqnarray}

Now, any monic cubic polynomial with negative $x^2$ coefficient can be transformed via a substitution $x\rightarrow x+n, n\in \mathbb{N}$ to another monic polynomial having $x^2$ coefficient $-1, 0$, or $1$ (we will call such a polynomial \textit{reduced}).  If $\alpha$ is a root of the latter, then $\alpha +n$ is a root of the former.  Thus in order to prove our result it suffices to show that, given any reduced cubic polynomial $p(x)$, there is an interesting factor $g(x)$ and natural number $n$ such that $p(x)=g(x+n)$.

We will proceed with each of the three types of reduced polynomial in turn, showing that for each type, and for every choice of the $x-$coefficient and constant term, the $6$ parameters in the $(3,k)$-biclique construction can be chosen in such a way as to produce the desired chromatic polynomial.  There are no doubt many possible ways in which to correctly choose the parameters; in each case we will mention just one.

\subsubsection*{Case 1: $a_2=-1$}
Let $p(x)=x^3-x^2+a_1x+a_0$, and let $i$ represent any number.  Assign the below values to the parameters $a,b,c,d,e,f$:
\begin{eqnarray*}
a&=&(2n+a_0)^2-11a_0+35+a_1-(8a_0-45)i-(16i+24)n+16i^2\\
b&=&-2i+n-3\\
c&=&(2n+a_0)^2-13a_0+46+a_1-(8a_0-53)i-(16i+28)n+16i^2\\
d&=&i+1\\
e&=&-(2n+a_0)^2+12a_0-41-a_1+(8a_0-50)i+(16i+27)n-16i^2\\
f&=&i
\end{eqnarray*}
Let $g(x)$ be the polynomial obtained by substituting these values into (\ref{cp}).  Then we have
\[g(x)=x^3+(-3n-1)x^2+(3n^2+2n+a_1)x-n^3-n^2-a_1n+a_0=p(x-n),\]
as desired. It remains to show that, for any $a_0$ and $a_1,$ appropriate values for $i$ and $n$ can be found such that each of the above parameters are non-negative integers. From the expressions for $b,d$ and $f$, $i$ must be non-negative and $n$ must satisfy $n\geq2i+3$.  We introduce a new variable $t$ by making the substitution
\[n=-a_0/2+2i+t,\]
giving us new expressions for $a,c$ and $e$:
\begin{eqnarray*}
a&=&a_0+35+a_1-3i-24t+4t^2\\
c&=&a_0+46+a_1-3i-28t+4t^2\\
e&=&-3a_0/2-41-a_1+4i+27t-4t^2
\end{eqnarray*}
Requiring that all these be non-negative then gives us the three inequalities:
\begin{eqnarray}
3i&\leq&a_0+35+a_1-24t+4t^2 \label{1}\\
3i&\leq&a_0+46+a_1-28t+4t^2 \label{2}\\
4i&\geq&3a_0/2+41+a_1-27t+4t^2 \label{3}
\end{eqnarray}
Let $t$ be an integer that is greater than $3$, greater than $a_0/2 +3$, and otherwise large enough to satisfy:
\[\frac{a_0+46+a_1-28t+4t^2}{3}\geq\frac{3a_0/2+41+a_1-27t+4t^2}{4}+1.\]
There is at least one integer between the expression on the left and that on the right.  Choose $i$ to be such an integer; then the chosen values for $i$ and $t$ satisfy (\ref{2}) and (\ref{3}).  Because $t\geq 3$, (\ref{2}) implies (\ref{1}).  Finally set $n=\lceil -a_0/2\rceil +2i+t.$ Because $t> a_0/2 +3$, $n$ then satisfies the condition $n\geq2i+3$.

The remaining two cases are similar, and so will be more briefly described.

\subsubsection*{Case 2: $a_2=0$}

Let $p(x)=x^3+a_1x+a_0x$, and again let $i$ be any number.  This time set:
\begin{eqnarray*}
a&=&(n+a_0)^2+a_1+14+19i+9i^2-(6i+8)n-(6i+6)a_0\\
b&=&-2i+n-3\\
c&=&(n+a_0)^2+a_1+20+25i+9i^2-(6i+10)n-(6i+8)a0\\
d&=&i+1\\
e&=&-(n+a_0)^2-a_1-18-23i-9i^2+(6i+10)n+(6i+7)a_0\\
f&=&i
\end{eqnarray*}
Let $g(x)$ be the polynomial obtained by substituting these values into (\ref{cp}).  Then
\[g(x)=x^3-3nx^2-(3n^2-a_1+3n^2)x-n^3-a_1n+a_0=p(x-n).\]
Now make the substitution
\[n=-a_0+3i+t.\]
This gives us the following expressions for $a,c$ and $e$:
\begin{eqnarray*}
a&=&t^2+a_1+14-5i+2a_0-8t\\
c&=&t^2+a_1+20-5i+2a_0-10t\\
e&=&-t^2-a_1-18+7i-3a_0+10t,
\end{eqnarray*}
leading to the inequalities:
\begin{eqnarray*}
5i&\leq&t^2+a_1+14+2a_0-8t\\
5i&\leq&t^2+a_1+20+2a_0-10t\\
7i&\geq&t^2+a_1+18+3a_0+10t.
\end{eqnarray*}
Again, by choosing $t$ to be very large, a positive value for $i$ can be found to satisfy these for any $a_0,a_1$.

\subsubsection*{Case 3: $a_2=1$}

Let $p(x)=x^3+x^2+a_1x+a_0x$, and set:
\begin{eqnarray*}
a&=&a_0^2+5-a_0+a1+(3-4a_0)i-2n+4i^2\\
b&=&-2i+n-3\\
c&=&a_0^2+6-3a_0+a_1+(7-4a_0)i-2n+4i^2\\
d&=&i+1\\
e&=&-a_0^2-7+2a_0-a_1-(6-4a_0)i+3n-4i^2\\
f&=&i
\end{eqnarray*}
Substituting into (\ref{cp}) we get
\[g(x)=x^3+(1-3n)x^2+(3n^2-2n+a1)x-n^3+n^2-a1n+a0=p(x-n).\]
We now express $i$ in terms of a new parameter $t$, setting:
\[i=a_0/2-t.\]
This gives us
\begin{eqnarray*}
a&=&5+a_0/2+a_1-3t-2n+4t^2\\
c&=&6+a_0/2+a_1-7t-2n+4t^2\\
e&=&-7-a_0-a_1+6t+3n-4t^2,
\end{eqnarray*}
and so we must satisfy
\begin{eqnarray*}
2n&\leq&5+a_0/2+a_1-3t+4t^2\\
2n&\leq&6+a_0/2+a_1-7t+4t^2\\
3n&\geq&7+a_0+a_1-6t+4t^2.
\end{eqnarray*}
This time we need to choose a large negative value for $t$.  If it is large enough then $d$ and $f$ will be non-negative, and we can easily find a positive $n$ to satisfy the three inequalities, as well as the requirement $n\geq 2i+3$.

Thus we have given a means to construct a biclique with a chromatic root $\alpha +n$ for any cubic integer $\alpha$, thereby proving the cubic case of the $\alpha +n$ conjecture.

Now, the quadratic case of the $\alpha +n$ conjecture was proved using a subfamily of a family of graphs known as \emph{rings of cliques}.  However, it is interesting to note that members of this subfamily have precisely the same construction as $(2,k)$-bicliques, meaning that bicliques have been used to satisfy the $\alpha +n$ conjecture in both cases.  Given the exponential increase in the number of these graphs as $j$ increases (constructed as above, an $(j,k)$-biclique has $2^j-2$ parameters), it seems entirely plausible that they might satisfy the general conjecture.  Unfortunately the increase in parameters leads to difficulties in finding correct specialisations in the manner of the two cases proved so far, and it seems likely that a different method from that used in this paper would need to be found for algebraic numbers of higher degree.

\subsubsection*{Acknowledgements}
This paper was written while under the supervision of Peter Cameron at Queen Mary, University of London.  I would like to thank Prof. Cameron for suggesting this topic, and for his helpful comments and advice.  I am also indebted to a referee for pointing out the connection between the chromatic polynomial of a biclique and the matchings of its complement; his or her insight led to a significant improvement in the exposition of this paper.

\bibliographystyle{spmpsci}
\bibliography{mybib}
\end{document}